\documentclass [11pt, reqno]{amsart}
\usepackage{times}
\usepackage{amssymb,latexsym,amsmath,amsfonts, eucal}
\newtheorem{thm}{Theorem}[section]
\newtheorem{cor}[thm]{Corollary}
\newtheorem{lem}[thm]{Lemma}
\newtheorem{prop}[thm]{Proposition}
\theoremstyle{definition}
\newtheorem{defn}[thm]{Definition}
\theoremstyle{remark}
\newtheorem{rem}[thm]{Remark}
\numberwithin{equation}{section}

\newcommand{\beas}{\begin{eqnarray*}}
\newcommand{\eeas}{\end{eqnarray*}}
\newcommand{\bes} {\begin{equation*}}
\newcommand{\ees} {\end{equation*}}
\newcommand{\be} {\begin{equation}}
\newcommand{\ee} {\end{equation}}
\newcommand{\bea} {\begin{eqnarray}}
\newcommand{\eea} {\end{eqnarray}}

\newcommand{\bC} {\mathbb{C}}

\title[Properties of singular integral operators $S_{\alpha,\beta}$]{Properties of singular integral operators $S_{\alpha,\beta}$}
\thanks{The first author is supported by the NBHM Postdoctoral Fellowship, Govt. of India. The second author is supported by the Feinberg Postdoctoral Fellowship of the Weizmann Institute of Science. }

\author[A. Samanta]{Amit Samanta}
\address{Amit Samanta, Department of Mathematics and Statistics, Indian Institute Of Technology, Kanpur-208016, India }
\email{{\tt amit.gablu@gmail.com, asamanta@iitk.ac.in}}

\author[S. Sarkar]{Santanu Sarkar}
\address{Santanu Sarkar, Department of Mathematics, The Weizmann Institute of Science, P.O. Box 26, Rehovot-7610001, Israel }
\email{{\tt santanu87@gmail.com, santanu.sarkar@weizmann.ac.il}}

\begin{document}

%\maketitle
{\begin{abstract} For $\alpha, \beta \in L^{\infty} (S^1),$  the singular integral operator $S_{\alpha,\beta}$ on $L^2 (S^1)$ is defined by $S_{\alpha,\beta}f:= \alpha Pf+\beta Qf$, where $P$ denotes the orthogonal projection of $L^2(S^1)$ onto the Hardy space $H^2(S^1),$ and $Q$ denotes the orthogonal projection onto $H^2(S^1)^{\perp}.$  In a recent paper  Nakazi and Yamamoto have studied the normality and self-adjointness of $S_{\alpha,\beta}.$ This work has shown that $S_{\alpha,\beta}$ may have analogous properties to that of the Toeplitz operator. In this paper we study several other properties of $S_{\alpha,\beta}.$ 
\end{abstract}


\maketitle
\section{Introduction}
Let $L^2 = L^2(S^1)$ denotes the the Hilbert space of square integrable functions on the circle $S^1 = \{ z \in \bC : |z| = 1\}$ with respect to the normalized Lebesgue measure. The inner product of two functions $f, g \in L^2 $ is given by $$ \langle f, g \rangle = \frac{1} {2\pi} \int_{0}^{2\pi} f(e^{i\theta}) \bar{g}(e^{i\theta})d\theta.$$ The norm will be denoted by $\parallel . \parallel,$  that is $\parallel f \parallel ^2 = \langle f , f \rangle.$ Let $L^{\infty} = L^{\infty} (S^1)$ denotes the space of all essentially bounded measurable functions on $S^1.$  The norm of a function $f \in L^{\infty} $ is given by $\parallel f \parallel _{\infty} = {\mbox{ess sup}}_{S^1} |f|.$ Let $
H^2$ denotes the usual Hardy space on $S^1.$ That is  it consists of all $f$ in $L^2$ with all the  negative Fourier coefficients equal to zero. Similarly we define $H^{\infty},$ to be the space of all $L^{\infty}$ functions with all the negative Fourier coefficients zero. Let $H_{2}^{\perp}$ denotes the orthogonal complement of $H^2$ in $L^2.$    Let $P$ and $Q$ denote the orthogonal projection of $L^2$ onto $H^2$ and $H^{2\bot}$ respectively. Thus $P + Q = I,$ where $I$ is the identity operator on $L^2.$ 

Let $S$ be the singular integral operator defined by  $$(Sf)(z) = \frac{1} {\pi i} \int_{S^1} \frac{f(w)}{w-z} dw.$$
    This operator is well-studied (\cite{GK}, Vol.I, p.12).  $S$ can be written in terms of $P$ and $Q:$  $S = P - Q.$   The operator $S$ has natural generalization to the operator $S_{\alpha, \beta}$ where $S_{\alpha, \beta} (\alpha,\beta\in L^\infty)$ on $L^2$ is defined by
$$S_{\alpha,\beta}f:= \alpha Pf+\beta Qf, \hspace{4mm}f\in L^2.
$$ Some properties of $S_{\alpha,\beta}$  related to the norm (or weighted norm), invertibility, boundedness (with weight) etc have been studied in  (\cite{NY1}, \cite{NY2} , \cite{NY3}, \cite{NY4}, \cite{NY5}, \cite{NY6}, \cite{NY7}, \cite{NY8}, \cite{Y1}, \cite{Y2}). Some of these results are generalizations of properties of $S.$ In a recent paper (\cite{NY8}) the normality and self-adjointness of the operator $S_{\alpha,\beta}$ are studied. This work has shown that $S_{\alpha,\beta}$ may have analogous properties to the Toeplitz operator. In this paper we found several analogous properties of $S_{\alpha,\beta}$ corresponding to that of the Toeplitz operator. 

The paper is organized as follows. 
In Section 2 we give a characterization of $S_{\alpha,\beta}$ in terms of its matrix representation. In Section 3 we discuss about its operator norm. In Section 4 we give characterization of $S_{\alpha,\beta}$ as a commutator of $S_{z,\bar{z}}.$ Section 5 deals with the invariant subspace and reducing subspace of $S_{z,\bar{z}}.$
In Section 6 we discuss about composition of two operators $S_{\alpha_1,\beta_1}$ and $S_{\alpha_2,\beta_2}.$ Also we study their commutativity. In Section 7 we give some results related to the compactness of $S_{\alpha,\beta}.$ Section 8 deals with the spectrum of $S_{\alpha,\beta}.$ In Section 9 we discuss about the injectivity of $S_{\alpha,\beta}$ and its adjoint. 

In order to compare our results with that of the Toeplitz operator the corresponding properties of the Toeplitz operator is worth mentioning. 
For $\phi \in L^\infty$ the Toeplitz operator $T_{\phi}$ on $H^2$ is defined by $T_{\phi} (f) = P(\phi f)$ for all $f \in H^2.$ The following mentioned properties of the Topelitz operator are well-known and can be found in (\cite{MR}, Chapter 3 and Chapter 1).
\begin{itemize}
\item  A bounded operator on $H^2$ is a Toeplitz operator if and only if its matrix with respect to the orthonormal basis $\{ e^{in\theta}\}$ has constant diagonal entries. That is $A_{m_1, n_1} = A_{m_2, n_2}$ whenever $m_1 - n_1 = m_2 - n_2$ where $(A_{m,n})_{m,n \geq 0}$ denotes the $(m,n)$th entry of the matrix.
    \item $\parallel T_{\phi} \parallel = \mbox{Spectral radius of } T_{\phi} = \parallel \phi \parallel _{\infty}.$
    \item The commutant of the unilateral shift acting on $H^2$ is $T_\phi$ such that $\phi \in H^{\infty}.$
    \item {\textit{(Beurling's Theorem).}} Every invariant subspace of the unilateral shift acting on $H^2$ other than $\{0\}$ has the form $\psi H^2,$ where $\psi$ is an inner function, i.e. $\psi \in H^{\infty}, |\psi| =  1$ a.e.
     \item The only reducing subspaces of the unilateral shift are $\{0\}$ and $H^2.$
     \item Let $\psi, \phi \in L^{\infty}$. Then $T_{\psi} T_{\phi}$ is a Toeplitz operator if and only if $\psi$ is co-analytic or $\phi$ is analytic. In both of these cases $T_{\psi} T_{\phi} = T_{\psi \phi}.$ It follows that the product of two Toeplitz operator is zero if and only if at least one of the factor is zero. Also it follows that if both $\phi$ and $\psi$ are analytic or both  $\phi$ and $\psi$ are co-analytic, then they commute with each other. In fact there can arise one more case when $T_{\phi}$ commutes with $T_{\psi}$ and that is $a \phi + b \psi$ is constant for some constants $a$ and $b$ not both equal to zero.
      \item The only compact Toeplitz operator is the zero operator. Toeplitz operator can not even get closer to compact operators, more preciously if $\phi \in L^{\infty}$ and $K$ is a compact operator then $\parallel T_{\phi} - K \parallel \geq \parallel T_{\phi} \parallel.$
          \item Spectral radius of the Toeplitz operator is $\parallel \phi \parallel _{\infty}.$ If $\phi$ is analytic or co-analytic then $\sigma(T_{\phi}) = \overline{\phi(\mathbb{D})},$ where $\mathbb{D}$ is the open unit disc in $\mathbb{C}.$ If $\phi$ is continuous (so that it can be cosidered as a closed curve in $\mathbb{C}$), then 
 $$
\sigma(T_{\phi}) = \mbox{Range} (\phi) \cup \{ a \in \mathbb{C}: a \notin \mbox{Range} (\phi)\, \mbox{and}~~ \mbox{ind}_{a} \phi \neq 0 \},
$$ where~~$ \mbox{ind}_{a} \phi = \frac{1} {2\pi i} \int_{\phi} \frac{1} {z-a} dz$.
          %, and in the integral sign, $\phi$ is thought of as a closed curve in $\mathbb{C}$ defined by its range.
              \item {\textit{(The Coburn Alternative).}} If $\phi$ is a non-zero function in $L^{\infty},$ then at least one of $T_{\phi}$ and $T_{\phi}^*$ is injective.
\end{itemize}

We end this section with a well-known theorem of F. and M. Riesz, which will be used several times in this paper.
\begin{thm}{\textit{(The F. and M. Riesz Theorem).}} If $f \in H^2$ and the set $\{ e^{i \theta}: f(e^{i \theta}) = 0\}$ has positive measure, then $f$ is identically zero.
\end{thm}
\noindent It follows easily from the theorem that the same is true if $f$ is a co-analytic function in $L^2.$

%\newpage
\section{Matrix}
The following theorem gives a characterization $S_{\alpha,\beta}$ in terms of its matrix representation.

%The following theorem characterizes an operator $S_{\alpha,\beta}$ in terms of its matrix entries.
\begin{thm}
Let $T$ be a bounded linear operator on $L^2$. Then $T=S_{\alpha,\beta}$ for some $L^\infty$ functions $\alpha$ and $\beta$ iff its matrix with respect to the orthonormal basis $\{z^n\}_{n=-\infty}^\infty$ has the form
$$
\begin{bmatrix}
\ddots &\ddots &\ddots\\
\ddots &b_{-1}&b_{-2}&b_{-3}\\
\ddots &b_0&b_{-1}&b_{-2}&a_{-3}\\
\ddots &b_1&b_0&b_{-1}&a_{-2}&a_{-3}\\
\ddots &b_2&b_1&b_0&a_{-1}&a_{-2}&a_{-3}\\
&b_3&b_2&b_1&\textbf{a}_0&a_{-1}&a_{-2}&a_{-3}\\
&&b_3&b_2&a_1&a_0&a_{-1}&a_{-2}&\ddots\\
&&&b_3&a_2&a_1&a_0&a_{-1}&\ddots\\
&&&&a_3&a_2&a_1&a_0&\ddots\\
&&&&&a_3&a_2&a_1&\ddots\\
&&&&&&\ddots &\ddots &\ddots
\end{bmatrix}
$$
for some constants $a_n$ and $b_n$, $n\in\mathbb{Z}$ (the set of all integers); more precisely its $(m,n)$th entry is $a_{m-n}$ or $b_{m-n}$ accordingly whether $n\geq 0$ or $n\leq -1$. (In the matrix boldface denotes the $(0,0)$ position).
In this case $a_n=\hat\alpha(n)$ and $b_n=\hat\beta(n)$.
\end{thm}

\begin{proof}
First assume that $T=S_{\alpha,\beta}$ for some $\alpha,\beta\in L^\infty$. Since $Tz^n=\alpha z^n$ if $n\geq 0$ and $Tz^n=\beta z^n$ if $n\leq -1$, it follows that $\langle Tz^n,z^m\rangle$ is $\hat\alpha(m-n)$ or $\hat\beta(m-n)$ accordingly whether $n\geq 0$ or $n\leq -1$. Therefore the matrix of $T$ with respect to the orthonormal basis $\{z^n\}_{n=-\infty}^{\infty}$ has the given form with $a_n=\hat\alpha(n)$ and $b_n=\hat\beta(n)$. This proves the necessary part.

For the converse part, assume that the matrix of $T$ has the given form. Define $\alpha:=T1$ and $\beta=zTz^{-1}$. Then $\alpha,\beta\in L^2$, and
$$
\alpha(z)=(T1)(z)=\sum_{m=-\infty}^{\infty}\langle T1,z^m\rangle z^m=\sum_{m=-\infty}^{\infty}a_mz^m
,$$
 since $\langle T1,z^m\rangle =(m,0)$ th entry in the matrix $=a_m$.  Similarly
$$
 \beta(z)=zTz^{-1}=z\bigg(\sum_{m=-\infty}^{\infty}\langle Tz^{-1},z^m\rangle z^m\bigg)=z\bigg(\sum_{m=-\infty}^\infty b_{m+1}z^m\bigg)=\sum_{m=-\infty}^\infty b_{m}z^m
 .$$
If $n\geq 0$,
$$
Tz^n=\sum_{m=-\infty}^{\infty}\langle Tz^n,z^m\rangle z^m=\sum_{m=-\infty}^{\infty}a_{m-n}z^m=z^n\bigg(\sum_{m=-\infty}^{\infty}a_{m-n}z^{m-n}\bigg)=z^n\alpha(z).
$$
In a similar way, $Tz^n=z^n\beta(z)$ if if $n\leq -1$. Since $T$ is linear, we can say that $Tf=\alpha f$ if $f$ is a trigonometric polynomial in $H^2$, and $Tf=\beta f$ if $f$ is a trigonometric polynomial in $H^{2\bot}$. Now, if $f\in H^2$, then there is a sequence ${f_n}$ of trigonometric polynomials in $H^2$ such that $f_n\rightarrow f$ in $L^2$, so that $\alpha f_n=Tf_n\rightarrow Tf$ in $L^2$, since $T$ is bounded. Now along a subsequence $f_{n_k}\rightarrow f$ point wise almost everywhere. Hence $\alpha f_{n_k}\rightarrow \alpha f$ point wise almost everywhere. So we can conclude that $Tf=\alpha f$ almost everywhere  for all $f \in H^2.$ In a similar way we can show that $Tf=\beta f$ for all $f$ in $H^{2\bot}$. So, we have $Tf=\alpha Pf+\beta Qf$ for all $f\in L^2$ or $T=S_{\alpha,\beta}$. It remains to prove that $\alpha$ and $\beta$ are in $L^\infty$. But this proof  is standard and we skip the proof here, because in another occasion we shall  prove the same thing(see the last part of the proof of Theorem \ref{commutators theorem} in section 4).
\end{proof}

\section{Operator norm}
Let $\alpha,\beta\in L^\infty$. In \cite{NY3} (Theorem 2.1), Nakazi, Yamamoto has given the following formula for the operator norm of $S_{\alpha,\beta}$ :
$$
||S_{\alpha,\beta}||=\inf_{k\in H^\infty}\left|\left|\frac{|\alpha|^2+|\beta|^2}{2}+\sqrt{|\alpha\bar\beta-k|^2+\left(\frac{|\alpha|^2-|\beta|^2}{2}\right)^2}\right|\right|_{\infty}
$$
We shall use some of the following  results related to the operator norm in future. Though they can be deduced from the above formula, we shall give direct simple proofs. 
For $\beta\in L^\infty$, we define the operator $\tilde{T_\beta}$ on $H^{{2}^{\perp}}$ given by $$\tilde{T_\beta} (f)  = Q(\beta f), ~~~f \in H^{{2}^{\perp}}. $$ % We denote the space $L^\infty\cap H^2$ by $H^\infty$.
\begin{thm}\label{norm of S alpha beta}
Let $\alpha,\beta\in L^\infty$.

\noindent \textup{(i)} If $\alpha_n\rightarrow\alpha$ and $\beta_n\rightarrow\beta$ in $L^\infty$ norm then $S_{\alpha_n,\beta_n}\rightarrow S_{\alpha,\beta}$ in the operator norm.
%, in particular $||S_{\alpha_n,\beta_n}||\rightarrow ||S_{\alpha,\beta}||$.

\noindent \textup{(ii)} $\textup{max}\big\{||\alpha||_\infty,||\beta||_\infty\big\}\leq||S_{\alpha,\beta}||\leq\sqrt{||\alpha||^2_\infty+||\beta||^2_\infty}$.

\noindent \textup{(iii)} If $\alpha\bar{\beta}\in H^\infty$ then $||S_{\alpha\beta}||=\textup{max}\big\{||\alpha||_\infty,||\beta||_\infty\big\}$

\noindent \textup{(iv)} If $|\alpha|=|\beta|=$ constant and $\alpha\bar\beta\in\frac{H^{2\bot}}{H^2}$ i.e. $\alpha\bar\beta$ can be written as quotient of a function in $H^{2\bot}$ and a function in $H^2$, then $||S_{\alpha,\beta}||=\sqrt{||\alpha||^2_\infty+||\beta||^2_\infty}$ .

\noindent \textup{(v)} There exists $\alpha,\beta$ such that $\textup{max}\big\{||\alpha||_\infty,||\beta||_\infty\big\}<||S_{\alpha,\beta}||<\sqrt{||\alpha||^2_\infty+||\beta||^2_\infty}$.
\end{thm}
\begin{proof} (i) Proof follows from the fact that, for any $f\in L^2$,
\begin{eqnarray*}
||(S_{\alpha_n,\beta_n}-S_{\alpha,\beta}
)f||&=&||(\alpha_n-\alpha)Pf+(\beta_n-\beta)Qf||\leq ||\alpha_n-\alpha||_\infty||f||+ ||\beta_n-\beta||_\infty||f||.
\end{eqnarray*}
(ii)
\begin{eqnarray*}
||S_{\alpha,\beta}||&=&\sup_{f\in L^2,||f||=1}||S_{\alpha,\beta}f||\\&\geq &\sup_{f\in H^2,||f||=1}||\alpha f||\\&\geq &\sup_{f\in H^2,||f||=1}||P(\alpha f)||\\&= &\sup_{f\in H^2,||f||=1}||T_\alpha f||=||T_\alpha||.
\end{eqnarray*}
But we know that $||T_\alpha||=||\alpha||_\infty$. Therefore $||S_{\alpha,\beta}||\geq||\alpha||_\infty$. In a similar way, using the fact $||\tilde{T}_\beta||=||\beta||_\infty$ (the proof of this fact is similar to the proof of operator norm of the Toeplitz operator), we can prove that $||S_{\alpha,\beta}||\geq||\beta||_\infty$. Hence $||S_{\alpha,\beta}||\geq\textup{max}\big\{||\alpha||_\infty,||\beta||_\infty\big\}$.

For any $f\in L^2$,
\begin{eqnarray*}
||S_{\alpha,\beta}f||&\leq &||\alpha||_\infty||Pf||+||\beta||_\infty||Qf||\\&\leq &\big(||\alpha||_\infty^2+||\beta||_\infty^2\big)^{\frac{1}{2}}\big(||Pf||^2+||Qf||^2\big)^{\frac{1}{2}}\\&=& \sqrt{||\alpha||_\infty^2+||\beta||_\infty^2}||f||.
\end{eqnarray*}
Therefore $||S_{\alpha,\beta}||\leq\sqrt{||\alpha||^2_\infty+||\beta||^2_\infty}$.

(iii) Since $\alpha\bar\beta\in H^\infty$, for all $f\in L^2$ we have  $\langle \alpha Pf,\beta Qf\rangle=\langle\alpha\bar\beta,\overline{Pf}Qf\rangle=0.$ Thus
$$
||S_{\alpha,\beta}f||^2=||\alpha Pf||^2+||\beta Qf||^2\leq \textup{max}\{||\alpha||_\infty,||\beta||_\infty\}\big(||Pf||^2+||Qf||^2\big)=\textup{max}\{||\alpha||_\infty,||\beta||_\infty\}||f||^2.
$$
Therefore $||S_{\alpha,\beta}||\leq\textup{max}\{||\alpha||_\infty,||\beta||_\infty\}$ and hence, by (ii), $||S_{\alpha,\beta}||=\textup{max}\{||\alpha||_\infty,||\beta||_\infty\}$.

(iv) In view of (ii), it is enough to show that for some $f\in L^2$, $||S_{\alpha,\beta}f||^2=(||\alpha||^2_\infty+||\beta||^2_\infty)||f||^2$. By the given condition it follows that there is an $L^2$ function $f$ such that $\alpha Pf-\beta Qf=0$. Since $|\alpha|=|\beta|=$ constant, 
$$
\left(||\alpha||^2_\infty+||\beta||^2_\infty\right)||f||^2-||S_{\alpha,\beta} f||^2=||\alpha Pf-\beta Qf||^2=0.
$$ 
% In view of (ii), it is enough to show that for some $f\in L^2$, $||S_{\alpha,\beta}f||=\sqrt{2}||f||$. The given conditions imply that $\beta=\alpha z^m$ for some $m \geq 0.$ Thus $||S_{\alpha\beta}f||=||Pf+z^m Qf||$. Therefore by taking $f=1+\bar z^m$, we get the desired result.

(v) For each $c\in[0,1]$, define $\beta_c(z)=cz+(1-c)\bar z$. Note that $c\rightarrow \beta_c$ is continuous from $[0,1]$ into $L^\infty$. Therefore, by (i), $c\rightarrow ||S_{1,\beta_c}||$ is continuous. Now, by (iii), $||S_{1,\beta_0}||=1$, and by (iv), $||S_{1,\beta_1}||=\sqrt{2}$. Therefore there exists $a\in(0,1) $ such that $||S_{1,\beta_a}||$ is strictly in between $1$ and $\sqrt 2$. But note that, for any $c\in[0,1]$, max$\{||1||,||\beta_c||_\infty\}=1$ and $\sqrt{||1||_\infty+||\beta_c||_\infty}=\sqrt{2}$. In particular, we conclude that
$$
\textup{max}\big\{||1||_\infty,||\beta_a||_\infty\big\}<||S_{1,\beta_a}||<\sqrt{||1||^2_\infty+||\beta_a||^2_\infty}.
$$
\end{proof}

\section{Commutators}
This section deals with the commutators of the operator $S_{z,\bar z}.$
It is easy to check
that $S^*_{z,\bar z}S_{z,\bar z}=I$. So the operator $S_{z,\bar z}$ is an isometry. 
\begin{defn} A function in  $L^2$ is said to be analytic if it is in $H^2,$  and it is said to be co-analytic if its conjugate is analytic.
\end{defn}
\begin{prop}
$S_{\alpha,\beta}$ commutes with $S_{z,\bar z}$ iff $\alpha$ is analytic and $\beta$ is co-analytic.
\end{prop}
\begin{proof}
Let $\alpha$ be analytic and $\beta$ be co-analytic. $S_{z,\bar z}S_{\alpha,\beta}f=S_{z,\bar z}(\alpha Pf+\beta Qf)=z\alpha Pf + \bar{z} \beta Q f.$ On the other hand $S_{\alpha,\beta}S_{z,\bar z}f=S_{\alpha,\beta}(zPf+\bar zQf)=\alpha zPf + \bar{z} \beta Q f$. Therefore $S_{z,\bar z}S_{\alpha,\beta}=S_{\alpha,\beta}S_{z,\bar z}$. Conversely, let $S_{\alpha,\beta}$ commutes with $S_{z,\bar z}$. Then  $S_{z,\bar z}S_{\alpha,\beta} 1=S_{\alpha,\beta}S_{z,\bar z}1$ which gives $zP\alpha+\bar zQ\alpha=z\alpha=zP\alpha+zQ\alpha$ so that $\bar zQ\alpha=zQ\alpha$ or $Q\alpha=0$ or $\alpha$ is analytic. Again  $S_{z,\bar z}S_{\alpha,\beta}\bar z=S_{\alpha,\beta}S_{z,\bar z}\bar z$ which gives $zP(\beta \bar z)+\bar zQ(\beta \bar z)=\beta\bar z^2=\bar zP(\beta \bar z)+\bar zQ(\beta\bar z).$ Therefore we get $P(\beta\bar z)=0$ or $\beta\bar z\in H^{2\bot}$ or $\beta$ is co-analytic.
\end{proof}

\begin{thm}\label{commutators theorem}
Let $T$ be a bounded operator on $L^2$. Then $T=S_{\alpha,\beta}$ for some $\alpha,\beta\in L^\infty$ with $\alpha$ analytic and $\beta$ co-analytic iff the following two conditions hold :

\noindent\textup{(i)} $T$ commutes with $S_{z,\bar z}$,

\noindent\textup{(ii)} $T\big(H^2-\{0\}\big)$ intersects $H^2$ and $T\big(H^{2\bot}-\{0\}\big)$ intersects $H^{2\bot}$.
\end{thm}
\begin{proof}
For the only if part, $(i)$ follows from the previous Proposition. $(ii)$ is clear; in fact more is true that both $H^2$ and $H^{2\bot}$ are invariant under $T$. For the converse part we follow the usual method. Assume that $T$ be a bounded operator satisfying conditions (i) and (ii). First note that $S^n_{z,\bar z}1=z^n$ and $S^{n}_{z,\bar z}\bar{z}=\bar z^{n+1}$, $n\geq 0$. Therefore, defining the $L^2$ functions $\alpha:=T1$, $\beta:=zT\bar z$ and using the fact that $T$ commutes with $S_{z,\bar z}$, we get
$$
T z^n=TS^n_{z,\bar z}1=S^n_{z,\bar z}T1=S^n_{z,\bar z}\alpha,
$$
$$
T \bar{z}^{n+1}=TS^n_{z,\bar z}\bar z=S^n_{z,\bar z}T\bar z=S^n_{z,\bar z}(\bar z\beta), \hspace {3mm} n\geq 0.
$$
But, it is easy to check that, for any $f$, $S^n_{z,\bar z}f=z^nPf+\bar z^nQf$. Using this in the above we get
$$
Tz^n=z^nP\alpha+\bar z^nQ\alpha,
$$
$$
T\bar{z}^{n+1}=z^nP(\bar z\beta)+\bar z^nQ(\bar z\beta),\hspace{3mm}n\geq 0.
$$
Define $\tilde{f}(z)=f(\bar z)$. Then the above two equations are clearly equivalent to saying that
\begin{eqnarray}\label{4.1}
Tf=fP\alpha+\tilde{f}Q\alpha,\hspace{3mm}\textup{for all trigonometric polynomial}\hspace{1mm} f\in H^2,
\end{eqnarray}
\begin{eqnarray}\label{4.2}
Tf=\widetilde{zf}P(\bar z\beta)+zfQ(\bar z\beta),\hspace{3mm}\textup{for all trigonometric polynomial}\hspace{1mm} f\in H^{2\bot}.
\end{eqnarray}
Now using the boundedness of $T$ we shall show that \ref{4.1} is
true for all $f$ in $H^2$ and \ref{4.2} is true for all $f$ in
$H^{2\bot}.$ Let $f\in H^2$. Then there is a sequence of analytic
trigonometric polynomials $f_n$ such that $f_n\rightarrow f$ in
$L^2$.
By \ref{4.1}, $Tf_n=f_nP\alpha+\tilde f_nQ\alpha $. Now, along a subsequence $f_{n_k}\rightarrow f$ point wise almost everywhere. Therefore $\big(f_{n_k}P\alpha+\tilde f_{n_k}Q\alpha\big)\rightarrow (fP\alpha+\tilde fQ\alpha)$ point wise almost everywhere. But $T$ being bounded $Tf_n\rightarrow Tf$ in $L^2$. So we conclude that $Tf=fP\alpha+\tilde{f}Q\alpha$ almost everywhere. In a similar way we can prove that \ref{4.2} is true for $f$ in $H^{2\bot}.$ So we  got
\begin{eqnarray}\label{4.3}
Tf=fP\alpha+\tilde{f}Q\alpha,\hspace{3mm}\textup{for all }\hspace{1mm} f\in H^2,
\end{eqnarray}
\begin{eqnarray}\label{4.4}
Tf=\widetilde{zf}P(\bar z\beta)+zfQ(\bar z\beta),\hspace{3mm}\textup{for all}\hspace{1mm} f\in H^{2\bot}.
\end{eqnarray}
Now, by condition (ii), there is
a non zero function $f_0$ in $H^2$ such that $Tf_0\in H^2$.
Therefore for $f = f_0$ equation (\ref{4.3})  implies that $\tilde{f}
_0Q\alpha=0$ or $Q\alpha=0$ or $\alpha$ is analytic. In a similar
way, using the fact $T\big(H^{2\bot}-\{0\}\big)$ intersects
$H^{2\bot}$ in equation (\ref{4.4}), we get $P(\bar z\beta)=0$ or $\bar z\beta\in H^{2\bot}$ or $\beta$ is co-analytic. Therefore (\ref{4.3}) gives $Tf=\alpha f$ if $f\in H^2$; and \ref{4.4} gives $Tf=\beta f$ if $f\in H^{2\bot}$. So, we can write $Tf=\alpha Pf+\beta Qf$ for all $f\in L^2$. Hence $T=S_{\alpha,\beta}.$ It remains to prove that $\alpha,\beta\in L^\infty$. The proof of these facts are exactly similar to that for a Toeplitz operators. But for the shake of completeness we present the proof of $\beta\in L^\infty$. The proof that $\alpha\in L^\infty$ will be
similar. If $T$ is zero operator the result is trivial. So assume
that $T\neq 0$.
Define $\gamma=\frac{\beta}{||T||}$. Since $\beta\in L^2$ so is $\gamma \theta$. If $f\in H^{2\bot}$, $Tf=\beta f$ and hence $||\beta f||=||Tf||\leq||T||||f||$. Therefore $||\gamma f||\leq||f||$. Putting $f=\bar z$ we get $||\gamma||=||\gamma \bar z||\leq 1$. Again $||\gamma^2||=||\gamma^2\bar z||=||\gamma(\gamma \bar z)||\leq ||\gamma\bar z||\leq 1$. Using induction we can prove that $||\gamma^n||\leq 1$ for all positive integers $n$. Now we claim that $||\gamma||_\infty\leq 1$. If not there is an $\epsilon>0$ such that the set
$$
E:=\{e^{i\theta}:|\gamma(e^{i\theta})|\geq 1+\epsilon\}
$$
has positive measure. But then one can easily show that $||\gamma^n||\geq(1+\epsilon)^n|E|$, where $|E|$ denotes the measure of $E$. Since $||\gamma^n||\leq 1$ for all $n$, $|E|$ must be zero which is a contradiction. Therefore our claim that $||\gamma||_{\infty}\leq 1$ is true, and consequently $\beta$ is in $L^\infty$.
\end{proof}
\begin{rem}
We want to remark that for the if part in the previous theorem we need the condition (ii). It is possible to get an bounded operator $T$ which commutes with $S_{z,\bar z}$ but $T$ can not be written in the form $S_{\alpha,\beta}$ for any $\alpha,\beta\in L^\infty$. Here we provide one such example. Define $T$ on the orthonormal basis as : $Tz^n=z^{n}+\bar z^{n+1}$ if $n\geq 0$  and $Tz^n=0$ if $n$ is negative. Extend $T$ by linearity to the space of all trigonometric polynomials. It is not hard to see that, for a trigonometric polynomial $f$, $Tf=Pf+\bar{z}\widetilde{Pf}$, so that $||Tf||\leq ||Pf||+||\bar z\widetilde{Pf}||\leq 2||f||$, since $||\widetilde{Pf}||=||Pf||$. therefore we can extend $T$ as a bounded operator on $L^2$. We continue to call this extended bounded operator as $T$. Now we show that $T$ commutes with $S_{z,\bar z}$. It is enough to show this only on the orthonormal basis $\{z^n:n\in\mathbb{Z}\}$, where $\mathbb{Z}$ denotes the set of all integers. If $n\geq 0$, then $TS_{z,\bar z}z^n=Tz^{n+1}=z^{n+1}+\bar{z}^{n+2}=S_{z,\bar z}(z^n+\bar z^{n+1})=S_{z,\bar z}Tz^n$. On the other hand if $n$ is negative both $TS_{z,\bar z}z^n$ and $S_{z,\bar z}Tz^n$ are zero.
\end{rem}

Form the above remark we have seen that only the condition that $T$ commutes with $S_{z,\bar z}$ does not imply that $T$ is of the form $S_{\alpha,\beta}$. But if we increase the set of commutators a little more it is possible to get the desired result with out imposing the condition (ii). In fact the condition (ii) will hold automatically in that case.

\begin{thm}
A bounded operator $T$ commutes with both $S_{z,0}$ and $S_{0,\bar z}$ iff $T=S_{\alpha,\beta}$ for some $\alpha , \beta \in L^\infty$ where $\alpha$ is analytic and $\beta$ is co-analytic.
\end{thm}
\begin{proof}
If $\alpha,\beta\in L^\infty$ with $\alpha$ analytic and $\beta$ co-analytic then $S_{z,0}S_{\alpha,\beta}f=z\alpha Pf$ and $S_{\alpha,\beta}S_{z,0}f=\alpha zPf$ for all $f$ in $L^2$ so that $S_{z,0}T=TS_{z,0}$. Similarly we can show that $T = S_{\alpha,\beta}$ commutes with $S_{0,\bar z}.$ Conversely, let $T$ commutes with both $S_{z,0}$ and $S_{0,\bar z}$. In view of the previous theorem it is enough to show that $T$ satisfies the conditions (i) and (ii) there. Since $S_{z,\bar z}=S_{z,0}+S_{0,\bar z}$, clearly $T$ commutes with $S_{z,\bar z}$ and thus satisfies condition (i). Now $S_{0,\bar z}T1=TS_{0,\bar z}1$ gives $\bar zQ(T1)=0$ or $Q(T1)=0$ or $T1\in H^2$. Again, $S_{z,0}T\bar z=TS_{z,0}\bar z$ gives $zP(T\bar z)=0$ or $P(T\bar z)=0$ or $T\bar z$ is in $H^{2,\bot}$. Hence $T$ satisfies condition (ii)  and we are done.
\end{proof}
\section{Invariant suspaces}
In this section we give description of invariant subspaces and reducings subspace of $S_{z,\bar z}$.

\begin{thm}
$M\subset L^2$ is an invariant subspace of $S_{z,\bar z}$ iff there exist inner functions $\phi$ and $\psi$ such that $M=\phi H^2\oplus\bar \psi H^{2\bot}.$
\end{thm}
\begin{proof}
Let $M$ be an invariant subspace of $S_{z,\bar z}$. Write $M=M_1\oplus M_2$, where $M_1\subset H^2$ and $M_2\subset H^{2\bot}$.  Note that $S_{z,\bar z}M=zM_1\oplus \bar z M_2$ which is contained in $M=M_1\oplus M_2$. Therefore $zM_1\subset M_1$ and $\bar zM_2\subset M_2$ or equivalently both  $M_1$ and $\bar{z} \bar{M_2}$ are forward shift invariant subspace of $H^2.$  Therefore, by Beurling`s Theorem there exists inner functions $\phi$ and $\psi$ such that $M_1=\phi H^2$ and $\bar z\bar M_2=\psi H^2$ or $ M_2=\bar\psi \bar z\bar H^2=\bar\psi H^{2\bot}$. Conversely, let $\phi$ and $\psi$  be two inner functions. By Beurling`s theorem $\phi H^2$ is invariant under the forward shift. Since $\phi H^2$ is a subspace of $H^2$ and $S_{z,\bar z}\mid_{H^2}$ is nothing but the forward shift on $H^2$, we can say that $\phi H^2$ is invariant under $S_{z,\bar z}$. Again, $\psi z$ being an inner function, by Beurling`s Theorem $\psi zH^2$ is an invariant subspace under forward shift. Equivalently $\bar\psi H^{2\bot}=\bar\psi\overline{zH^2}$ is invariant under multiplication by $\bar{z}.$ Hence $\bar\psi H^{2\bot}$ is invariant under $S_{z,\bar z}$. Therefore $\phi H^2\oplus\bar \psi H^{2\bot}$ is an invariant subspace of $S_{z,\bar z}.$
\end{proof}

\begin{thm}
$M\subset L^2$ is a reducing subspace of $S_{z,\bar z}$ iff $M$ is $L^2$ or $\{0\}$ or $H^2$ or $H^{2\bot}$.
\end{thm}

\begin{proof}
The sufficient part is easy to see. For the necessary part, let $M$ is a reducing subspace of $S_{z,\bar z}$. Write $M=M_1\oplus M_2$, where $M_1$ and $M_2$ are subspaces of $H_2$ and $H^{2\bot}$ respectively. Since $M$ is invariant under $S_{z,\bar z}$, by the proof of the previous theorem, both $M_1$ and $\bar z\bar M_2$ are forward shift invariant subspaces of $H^2$. Again, we have $S^*_{z,\bar z}M\subset M_1\oplus M_2$ or $P(\bar zM)+Q(z M)\subset M_1\oplus M_2$ so that $P(\bar z M)\subset M_1$ and $Q(zM)\subset M_2$. But $P(\bar zM)=P(\bar zM_1)$ and $Q(z M)=Q(zM_2)$. Therefore $M_1$ is an invariant subspace under the adjoint of forward shift (on $H^2$) . Thus it is a reducing subspace of forward shift (on $H^2$) and hence it must be equal to $H^2$ or $\{0\}$. On the other hand, since $Q(z M_2)\subset M_2$, one can show that $P( \bar z\bar z \bar M_2)\subset \bar z \bar M_2$. In fact, if $f\in \bar M_2$ then $\bar f\in M_2$ so that $Q(z\bar f)\in M_2$ or $z\big(\bar f-\hat{\bar f}(-1)\bar z\big)\in M_2$ or   $\big(\bar f-\hat{\bar f}(-1)\bar z\big)\in \bar zM_2$ or $\big( f-\bar{\hat{\bar f}}(-1) z\big)\in  z\bar M_2$ or  $\bar z^2f-\bar{\hat{\bar f}}(-1) \bar z\in \bar z\bar M_2$ or $\bar z^2f-\hat{f}(1) \bar z\in \bar z\bar M_2$ which implies $P(\bar z^2f)\subset \bar z M_2$ as desired. Therefore, $\bar z\bar M_2$ is also a reducing subspace of the forward shift operator on $H^2$ and hence $\bar z\bar M_2=H^2$ or $\{0\}$ or equivalently $M_2=\bar z \overline{H^2}=H^{2\bot}$ or $\{0\}$. But we already had that $M_1=H^2$ or $\{0\}$. So we conclude that $M$ is $L^2$ or $\{0\}$ or $H^2$ or $H^{2\bot}$.
\end{proof}

\section{Composition of two operators}
This section deals with composition of two operators of the form $S_{\alpha_1,\beta_1}$ and $S_{\alpha_2,\beta_2}.$ We shall show that when such a composition is again of the form $S_{\alpha, \beta}.$ We also study their commutativity.

\begin{thm}\label{thm-multiplicati is of the same form}
Let $\alpha_1,\beta_1,\alpha_2,\beta_2\in L^\infty$. Then $S_{\alpha_1,\beta_1}S_{\alpha_2,\beta_2}=S_{\alpha,\beta}$ for some $\alpha,\beta\in L^\infty$ iff either $\alpha_1=\beta_1$, or $\alpha_2$ is analytic and $\beta_2$ is co-analytic. In that case $\alpha=\alpha_1\alpha_2$ and $\beta=\beta_1\beta_2$.
\end{thm}
\begin{proof}
For $f\in L^2$,
$$
S_{\alpha_1,\beta_1}S_{\alpha_2,\beta_2}f=\alpha_1P(\alpha_2Pf+\beta_2Qf)+\beta_1Q(\alpha_2Pf+\beta_2Qf).
$$
If $\alpha_1=\beta_1$, then right hand side becomes $\alpha_1\alpha_2Pf+\beta_1\beta_2Qf$ which is nothing but $S_{\alpha_1\alpha_2,\beta_1\beta_2}f$. If $\alpha_2$ is analytic and $\beta_2$ is co-analytic, then too, the right hand side equals to $\alpha_1\alpha_2Pf+\beta_1\beta_2Qf=S_{\alpha_1\alpha_2,\beta_1\beta_2}$. This proves the if part of the Theorem. For the converse part, let
$S_{\alpha_1,\beta_1}S_{\alpha_2,\beta_2}=S_{\alpha,\beta}$ for some $\alpha,\beta\in L^\infty$. Applying on  $z^n$ to both sides, we get
\begin{eqnarray}\label{6.1}
\alpha_1P(\alpha_2 z^n)+\beta_1 Q(\alpha_2 z^n)=\alpha z^n, n\geq 0,
\end{eqnarray}
 \begin{eqnarray}\label{6.2}
\alpha_1P(\beta_2 z^n)+\beta_1 Q(\beta_2 z^n)=\beta z^n, n\leq -1,
\end{eqnarray}
Writing $P(\alpha_2 z^n)=\alpha_2 z^n-Q(\alpha_2 z^n)$ in \ref{6.1}, it follows that
$$
(\alpha-\alpha_1\alpha_2)z^n=(\beta_1-\alpha_1)Q(\alpha_2 z^n), n\geq 0.
$$
Taking $L_2$ norm of both side we get
$$
||\alpha-\alpha_1\alpha_2||=||(\beta_1-\alpha_1)Q(\alpha_2 z^n)||\leq ||\beta_1-\alpha_1||_{\infty}||Q(\alpha_2 z^n)||, n\geq 0.
$$
Since $||Q(\alpha_2 z^n)||\rightarrow 0$ as $n\rightarrow\infty$, we conclude that $\alpha=\alpha_1\alpha_2$. In a similar way, using \ref{6.2}, we can show that $\beta=\beta_1\beta_2$. Therefore \ref{6.1} with $n=0$ and \ref{6.2} with $n=-1$ respectively give $\alpha_1P(\alpha_2)+\beta_1Q(\alpha_2)=\alpha_1\alpha_2$ and $\alpha_1P(\beta_2\bar z)+\beta_1Q(\beta_2\bar z)=\beta_1\beta_2\bar z$. Writing $P(\alpha_2)=\alpha_2-Q(\alpha_2)$ and $Q(\beta_2\bar z)=\beta_2\bar z-P(\beta_2\bar z)$, it follows that $(\beta_1-\alpha_1)Q(\alpha_2)=0$ and $(\alpha_1-\beta_1)P(\beta_2\bar z)=0$. Since a non zero analytic or co-analytic function can not vanishes on a set of positive measure, either $\alpha_1=\beta_1$ or $\alpha_2$ analytic and $\beta_2$ co-analytic.
\end{proof}
\begin{cor}
Let $\alpha_1,\beta_1,\alpha_2,\beta_2\in L^{\infty}$. Then $S_{\alpha_1,\beta_1}S_{\alpha_2,\beta_2}=0$ iff at least one of the following holds.

\noindent (i) $\alpha_1=\beta_1$, $\alpha_1\alpha_2=\beta_1\beta_2=0$, 

\noindent (ii) $\alpha_1=\beta_2=0$, $\alpha_2$ is analytic. 

\noindent (iii) $\alpha_2=\beta_1=0$,  $\beta_2$ is co-analytic.

\noindent (iv) $\alpha_2=\beta_2=0$.

\end{cor}
\begin{proof}
If part is easy to see from the formula of $S_{\alpha_1,\beta_1}S_{\alpha_2,\beta_2}f$ :
$$
S_{\alpha_1,\beta_1}S_{\alpha_2,\beta_2}f=\alpha_1P(\alpha_2Pf+\beta_2Qf)
+\beta_1Q(\alpha_2Pf+\beta_2Qf).
$$
Conversely, let $
S_{\alpha_1,\beta_1}S_{\alpha_2,\beta_2}=0=S_{0,0}$. 

\noindent Case-1 : $\alpha_1=\beta_1$. Then $S_{\alpha_1\alpha_2,\beta_1\beta_2}=S_{\alpha_1,\beta_1}S_{\alpha_2,\beta_2}=0$ and hence $\alpha_1\alpha_2=\beta_1\beta_2=0$.

\noindent Case-2 : $\alpha_1\neq\beta_1$. Then by the previous theorem, $\alpha_2$ is analytic and $\beta_2$ is co-analytic. Also, $\alpha_1\alpha_2=0$ and $\beta_1\beta_2=0$. Since a non zero analytic (or co-analytic) function can not be zero on a set of positive measure, it follows that at least one of $\alpha_1$ and $\alpha_2$ is zero and at least one of $\beta_1$ and $\beta_2$ is zero. Since we are dealing with the case $\alpha_1\neq\beta_1$, at least one of $(ii)$, $(iii)$ and $(iv)$ must hold.
\end{proof}
\begin{thm}
Let $\alpha_1,\beta_1,\alpha_2,\beta_2\in L^\infty$. $S_{\alpha_1,\beta_1}$ commutes with $S_{\alpha_2,\beta_2}$ iff at least one of the following holds:

\noindent \textup{(i)} $\alpha_1,\alpha_2$ analytic and $\beta_1,\beta_2$ co-analytic.

\noindent \textup{(ii)} $\alpha_1=\beta_1$ and $\alpha_2=\beta_2$.

\noindent \textup{(iii)} There are constants $a,b,c$ with at least one of $a$ and $b$ is non-zero such that $a\alpha_1+b\alpha_2=a\beta_1+b\beta_2=c$.
\end{thm}
To prove the theorem, we need several lemmas.

\begin{lem}\label{lemma-1 to prove commutativity}
Let $\alpha_1,\beta_1,\alpha_2,\beta_2\in L^\infty$. $S_{\alpha_1,\beta_1}$ commutes with $S_{\alpha_2,\beta_2}$ iff following two holds:
\begin{eqnarray}\label{equation-1 in lemma-2}
(\alpha_1-\beta_1)Q(\alpha_2 f)=(\alpha_2-\beta_2)Q(\alpha_1 f)\hspace{3mm}\textup{for all}\hspace{1mm} f\in H^2,
\end{eqnarray}
\begin{eqnarray}\label{equation-2 in lemma-2}
(\alpha_1-\beta_1)P(\beta_2 g)=(\alpha_2-\beta_2)P(\beta_1 g)\hspace{3mm}\textup{for all}\hspace{1mm} g\in H^{2\bot}.
\end{eqnarray}
\end{lem}

\begin{proof}
$S_{\alpha_1,\beta_1}$ commutes with $S_{\alpha_2,\beta_2}$ iff $S_{\alpha_1,\beta_1}S_{\alpha_2,\beta_2}f=
S_{\alpha_2,\beta_2}S_{\alpha_1,\beta_1f}$ for all $f$ in $H^2$ and $S_{\alpha_1,\beta_1}S_{\alpha_2,\beta_2}g=
S_{\alpha_2,\beta_2}S_{\alpha_1,\beta_1}g$ for all $g$ in $H^{2\bot}$. Now, for $f\in H^2$,
\begin{eqnarray*}
S_{\alpha_1,\beta_1}S_{\alpha_2,\beta_2}f=
S_{\alpha_2,\beta_2}S_{\alpha_1,\beta_1f}&\Longleftrightarrow & S_{\alpha_1,\beta_1}(\alpha_2 f)=S_{\alpha_2,\beta_2}(\alpha_1 f)\\&\Longleftrightarrow & \alpha_1P(\alpha_2f)+\beta_1Q(\alpha_2 f)=\alpha_2P(\alpha_1f)+\beta_2Q(\alpha_1f)
\end{eqnarray*}
which is equivalent to (\ref{equation-1 in lemma-2}) since $\alpha_1P(\alpha_2f)=
\alpha_1\alpha_2f-\alpha_1Q(\alpha_2f)$ and $\alpha_2P(\alpha_1f)=
\alpha_1\alpha_2f-\alpha_2Q(\alpha_1f)$. In a similar way we can show that, for $g$ in $H^{2\bot}$, $S_{\alpha_1,\beta_1}S_{\alpha_2,\beta_2}g=
S_{\alpha_2,\beta_2}S_{\alpha_1,\beta_1}g$ is equivalent to (\ref{equation-2 in lemma-2}). Hence the proof.
\end{proof}

\begin{lem}\label{lemma-2 to prove commutativity}
Let $\phi_1,\psi_1,\phi_2,\psi_2\in L^\infty$. Then $\psi_1Q(\phi_2f)=\psi_2Q(\phi_1f)$ for all $f$ in $H^2$ iff at least one of the following holds :

\noindent\textup{(a)} Both $\phi_1$ and $\phi_2$ are analytic.

\noindent\textup{(b)} $\psi_1=\psi_2=0$.

\noindent \textup{(c)} There are constants $a$ and $b$ with at least one of them is non-zero such that  $a\psi_1+b\psi_2=0$ and $aQ(\phi_1)+bQ(\phi_2)=0$.
\end{lem}
\begin{proof}
In this proof, without further mentioning, we use the fact that a non zero analytic function can not vanish on a set of positive measure. It is obvious that any one of conditions (a) and (b) is sufficient, where as sufficiency of condition (c) follows from the fact that $Q(\phi_1)=cQ(\phi_2)$ implies $Q(\phi_1 f)=cQ(\phi_2 f)$ for all $f$ in $H^2$. For the necessary part, it is enough to prove that, if atleast one of $\phi_1$ and $\phi_2$ is not analytic then (b) or (c) holds. Here we prove this assuming that $\phi_2$ is not analytic. The proof when $\phi_1$ is not analytic will be similar. Two cases may happen : $\psi_1\psi_2=0$ and $\psi_1\psi_2\neq 0$. Assume that $\psi_1\psi_2=0$. The given equation with $f=1$ gives $\psi_1Q(\phi_2)=\psi_2Q(\phi_1)$. Multipling both sides with $\psi_1$ we get $\psi^2_1Q(\phi_2)=\psi_1\psi_2 Q(\phi_1)$ which implies $\psi^2_1=0$ or $\psi_1=0$. Similarly, multiplying both sides by $\psi_2$, we get $\psi_2=0$. So the first case implies condition (b). Now we consider the second case i.e. $\phi_2$ is not analytic and $\psi_1\psi_2\neq 0$. For any $f,g$ in $H^2$, we have $\psi_1Q(\phi_2f)=\psi_2Q(\phi_1f)$ and $\psi_1Q(\phi_2g)=\psi_2Q(\phi_1g)$. Cross multipliying these two equation we get
$$
\psi_1\psi_2Q(\phi_2f)Q(\phi_1g)=\psi_1
\psi_2Q(\phi_1f)Q(\phi_2g)
$$
or
\begin{eqnarray}\label{6.4.1}
Q(\phi_2f)Q(\phi_1g)=Q(\phi_1f)Q(\phi_2g) \hspace{3mm}\textup{for all}\hspace{1mm}f,g\in H^2.
\end{eqnarray}
Since $\phi_2$ is not analytic there exist integer $n_0\geq 1$ such that $\hat\phi_2(-n_0)\neq 0$. In \ref{6.4.1}, putting $f=z^{n_0-1}$ and $g=z^n$ ($n\geq 0$), we get $Q(\phi_2z^{n_0-1})Q(\phi_1 z^n)=Q(\phi_1z^{n_0-1})Q(\phi_2 z^n)$. Now we compare the coefficients of $z^{-2}$ from both sides.  Coefficient of $z^{-2}$ in the series of $Q(\phi_2z^{n_0-1})Q(\phi_1 z^n)$ is equal to the multiplication of the coefficient of $z^{-1}$ in the series of $Q(\phi_2z^{n_0-1})$ and the same in the series of $Q(\phi_1 z^n)$ which is nothing but $\hat\phi_2(-n_0)\hat\phi_1(-n-1)$. Similarly, the coefficient of $z^{-2}$ in the series of $Q(\phi_1z^{n_0-1})Q(\phi_2 z^n)$ is $\hat\phi_1(-n_0)\hat\phi_2(-n-1)$. So, we get $\hat\phi_2(-n_0)\hat\phi_1(-n-1)=\hat\phi_1(-n_0)\hat\phi_2(-n-1)$ or $\hat{\phi}_1(-n-1)=\frac{\hat\phi_1(-n_0)}{\hat\phi_2(-n_0)}\hat\phi_2(-n-1)$ for all $n\geq 0$. But this is equivalent to saying that $Q(\phi_1)=cQ(\phi_2)$, where $c=\frac{\hat\phi_1(-n_0)}{\hat\phi_2(-n_0)}$. Therefore the given equation with $f=1$ gives $\psi_1=c\psi_2$ and hence condition (c) follows.
\end{proof}
In a similar way we can prove the following lemma.
\begin{lem}\label{lemma-3 prove of commutativity}
Let $\phi_1,\psi_1,\phi_2,\psi_2\in L^\infty$. Then $\psi_1P(\phi_2g)=\psi_2P(\phi_1g)$ for all $g$ in $H^{2\bot}$ iff at least one of the following holds :

\noindent\textup{(a)} Both $\phi_1$ and $\phi_2$ are co-analytic.

\noindent\textup{(b)} $\psi_1=\psi_2=0$.

\noindent \textup{(c)} There are constants $a$ and $b$ with at least one of them non zero such that  $a\psi_1+b\psi_2=0$ and $aP(\bar z\phi_1)+bP(\bar z\phi_2)=0$.
\end{lem}

\textit{Proof of Theorem 6.3} : For the sufficient part we need to show that any one condition of (i), (ii) and (iii) implies equations (\ref{equation-1 in lemma-2}) and (\ref{equation-2 in lemma-2}) in Lemma \ref{lemma-1 to prove commutativity}. That this is true for condition (i) or (ii) is obvious. On the other hand, condition (iii) implies that $a(\alpha_1-\beta_1)=-b(\alpha_2-\beta_2)$, $aQ(\alpha_1)=-bQ(\alpha_2)$ and $aP(\bar z\beta_1)=-bP(\bar z\beta_2)$. But $aQ(\alpha_1)=-bQ(\alpha_2)$ is equivalent to saying that $aQ(\alpha_1f)=-bQ(\alpha_2f)$ for all $f\in H^2$, where as $aP(\bar z\beta_1)=-bP(\bar z\beta_2)$ is equivalent to saying that $aP(\beta_1g)=-bP(\beta_2g)$ for all $g\in H^{2\bot}$. Hence equations (\ref{equation-1 in lemma-2}) and (\ref{equation-2 in lemma-2}) follows. So the sufficient part is proved.

Now we prove the necessary part. So let $S_{\alpha_1,\beta_1}$ commutes with $S_{\alpha_2,\beta_2}$. By Lemma \ref{lemma-1 to prove commutativity}, equations (\ref{equation-1 in lemma-2}) and (\ref{equation-2 in lemma-2}) are true. But, equation (\ref{equation-1 in lemma-2}), by Lemma \ref{lemma-2 to prove commutativity}, implies that at least one of the following holds :

\noindent (1) Both $\alpha_1$ and $\alpha_2$ are analytic

\noindent (2) $\alpha_1=\beta_1$ and $\alpha_2=\beta_2$

\noindent (3) There are constants $a^\prime$ and $b^\prime$ with at least one of them non zero such that $a^\prime(\alpha_1-\beta_1)+b^\prime(\alpha_2-\beta_2)=0$ and  $a^\prime Q(\alpha_1)+b^\prime Q(\alpha_2)=0$.

\noindent Again, \ref{equation-2 in lemma-2}, by Lemma \ref{lemma-3 prove of commutativity}, implies that at lest one of the following holds :

\noindent (1$^\prime$) Both $\beta_1$ and $\beta_2$ are co-analytic

\noindent (2$^\prime$) $\alpha_1=\beta_1$ and $\alpha_2=\beta_2$

\noindent (3$^\prime$) There are constants $a^{\prime\prime}$ and $b^{\prime\prime}$ with at least one of them non zero such that $a^{\prime\prime}(\alpha_1-\beta_1)+b^{\prime\prime}(\alpha_2-\beta_2)=0$ and  $a^{\prime\prime} P(\bar z\beta_1)+b^{\prime\prime} P(\bar z\beta_2)=0$.

Since (2) or (2$^\prime$) implies (ii), we only need to consider the following four cases.

Case-1 : (1) and (1$^\prime$) are true. But this is nothing but condition (i).

Case-2 : (1) and (3$^\prime$) are true. (3$^\prime$) implies that $a^{\prime\prime}\alpha_1+b^{\prime\prime}\alpha_2=
a^{\prime\prime}\beta_1+b^{\prime\prime}\beta_2$. Since, by (1), $\alpha_1$ and $\alpha_2$ are analytic,
$a^{\prime\prime}\beta_1+b^{\prime\prime}\beta_2$ is analytic so that all of its negative Fourier coefficients are zero. But, all of its (strictly) positive Fourier coefficients are zero too since, by (3$^\prime$),  $P\big(\bar z(a^{\prime\prime}\beta_1+b^
{\prime\prime}\beta_2)\big)=0$. Therefore $a^{\prime\prime}\beta_1+b^{\prime\prime}\beta_2$ is nothing but a constant. So we get (iii) in this case.

Case-3 : (3) and (1$^\prime$) are true. This implies (iii). Proof is similar to the previous case.

Case-4 : (3) and (3$^\prime$) are true. We can assume that $\alpha_1\neq\beta_1$ or $\alpha_2\neq\beta_2$, because otherwise it will give (ii). So, without loss of generality assume that $\alpha_1\neq\beta_1$. This
will force both $b^\prime$ and $b^{\prime\prime}$ to be non zero and $a^{\prime}/b^\prime=a^{\prime\prime}/b^{\prime\prime}(=
d$ say). So, we have $d\alpha_1+\alpha_2=d\beta_1+\beta_2$, $Q(d\alpha_1+\alpha_2)=0$ and $P\big(\bar z(d\beta_1+\beta_2)\big)=0$. Therefore $P\big(\bar z(d\alpha_1+\alpha_2)\big)=0$. But, then $d\alpha_1+\alpha_2$ must be a constant and hence, this case implies (iii).

\section{Compactness}
In this section we shall discuss about the compactness of the operator $S_{\alpha,\beta}.$ The following theorem says that there are no non trivial compact operators $S_{\alpha,\beta}$. 
\begin{thm}
If $S_{\alpha,\beta}$ is compact then $\alpha=\beta=0$.
\end{thm}

\begin{proof}
Let $S_{\alpha,\beta}$ is compact. Note that $S_{\alpha,\beta}z^n=\alpha z^n$ for all $n\geq 0$, so that $||S_{\alpha,\beta}z^n||=||\alpha||$ if $n\geq 0$. Since $1, z, z^2\cdots$ is a sequence of orthonormal elements, $||S_{\alpha,\beta}z^n||\rightarrow 0$ as $n\rightarrow 0$. Therefore $||\alpha||=0$ or $\alpha=0$. In a similar way, taking the sequence $\bar z,\bar z^2,\bar z^3\cdots$, we can show that $\beta=0$.
\end{proof}
Compact operators can not even go closer to $S_{\alpha,\beta}$.
\begin{thm}
Let $\alpha,\beta\in L^\infty$, and $K$ be a compact operator on $L^2$. Then $||S_{\alpha,\beta}-K||\geq\frac{1}{\sqrt 2}||S_{\alpha,\beta}||$. The constant $\frac{1}{\sqrt 2}$, appeared in the inequality, is the maximum one.
\end{thm}
\begin{proof}
Let $n$ be a non-negative integer. Since the operator $S_{z^n,\bar z^n}$ sends $z^k$ to $z^{k+n}$ and $\bar z^{k+1}$ to $\bar z^{k+1+n}$ for $k\geq 0$, it follows that $S_{z^n,\bar z^n}$ is an isometry so that $||S^*_{z^n,\bar z^n}||=||S_{z^n,\bar z^n}||=1$. Therefore
\begin{eqnarray*}
||S_{\alpha,\beta}-K||&= & ||S^*_{\alpha,\beta}-K^*||\\&\geq& ||S^*_{z^n,\bar
z^n}(S^*_{\alpha,\beta}-K^*)||\\&= &||(S_{\alpha,\beta}S_{z^n\bar z^n})^*-
S^*_{z^n,\bar z^n}K^*||\\&=&||S^*_{\alpha z^n,\beta\bar z^n}-S^*_{z^n,\bar
z^n}K^*||\hspace{3mm}\textup{by}\hspace{1mm} \textup{Theorem}~\ref{thm-multiplicati is of the
same form}\\&\geq & ||S^*_{\alpha z^n,\beta\bar z^n}||-||S^*_{z^n,\bar z^n}
K^*||\\&=& ||S_{\alpha z^n,\beta\bar z^n}||-||S^*_{z^n,\bar z^n}K^*||.
\end{eqnarray*}
Now by Theorem \ref{norm of S alpha beta} (ii),
$$
||S_{\alpha z^n\beta\bar z^n}||\geq\textup{max}\{||\alpha||_\infty,||\beta||_\infty\}\geq\frac{1}{\sqrt 2}\sqrt{||\alpha||_\infty^2+||\beta||_\infty^2}\geq\frac{1}{\sqrt 2}||S_{\alpha,\beta}||.
$$
Again, for any $f$ in $L^2$, $S^*_{z^n,\bar z^n}f=P(\bar z^n f)+Q(z^n f)$ which goes to $0$ as $n\rightarrow \infty$. But $K$ being compact so is $K^*$. Therefore $||S^*_{z^n,\bar z^n}K^*||\rightarrow 0$. So we conclude the first part of the Theorem.

For the second part it is enough to show that there are $L^\infty$ functions $\alpha,\beta$ and a compact operator $K$ such that $||S_{\alpha,\beta}-K||=\frac{1}{\sqrt 2}||S_{\alpha,\beta}||$. Take $\alpha=\bar z$, $\beta=1$ and the finite rank operator (in particular compact) $K$ on $L^2$ defined by $Kf=\hat f(0)\bar z$. Then, it is easy to see that
$$
(S_{\bar z,1}-K)f=\sum_{n=1}^\infty\hat f(n)z^{n-1}+\sum_{n=-1}^{-\infty}\hat f(n)z^n
$$
so that
$$
||(S_{\bar z,1}-K)f||^2=\sum_{n\neq 0}|\hat f(n)|^2\leq ||f||^2.
$$
Again $||(S_{\bar z,1}-K)z||=1=||z||$. Therefore $||(S_{\bar z,1}-K)||=1$. On the other hand, by Theorem \ref{norm of S alpha beta} (iv), $||S_{\bar z,1}||=\sqrt 2$. Hence $||S_{\bar z,1}-K||=\frac{1}{\sqrt 2}||S_{\bar z,1}||$ as desired.
\end{proof}

\section{Spectrum}
In this section we discuss about the spectrum of the operator $S_{\alpha,\beta}.$

We can write
$$
S_{\alpha,\beta}=S_{\alpha,0}+S_{0,\beta}.
$$
It is easy to check that
$$
S^*_{\alpha,0}f=P(\bar{\alpha}f);\hspace{4mm}
S^*_{0,\beta}f=Q(\bar{\beta}f),
$$
so that
$$
S^*_{\alpha,\beta}f=P(\bar{\alpha}f)+Q(\bar{\beta}f).
$$
%For $\alpha\in L^\infty$, define the Toeplitz operator on $H^2$ by
%%$$
%T_\alpha f=P(\alpha f), \hspace{4mm}f\in H^2
%.$$
%Also, define $\tilde{T}_\alpha$ on $H^{2\bot}$ by
%$$
%\tilde{T}_\alpha f=Q(\alpha f), \hspace{4mm}f\in H^{2\bot}
%.$$
For an operator $T$, we denote its point spectrum by $\Pi_{0}(T),$ approximate point spectrum by $\Pi(T) $ and spectrum by $\sigma(T)$. First we prove the following lemma which will be useful later in discussing the spectrum.

\begin{lem}\label{l-1}
Let $\lambda\neq 0$. Then $(S_{\alpha,0}^*-\lambda I)f=g$ iff
\begin{eqnarray*}
Qf &=&-\frac{1}{\lambda}Qg,\\
(T_{\bar{\alpha}}-\lambda I)Pf &=& Pg+\frac{1}{\lambda}P(\bar\alpha Qg).
\end{eqnarray*}
In particular, $g\in H^2$ iff $f\in H^2$.
\end{lem}
\begin{proof}
Let $(S_{\alpha,0}^*-\lambda I)f=g$. Writing $f=Pf+Qf$ and $g=Pg+Qg$ we get
$$
P(\bar\alpha Pf)+P(\bar\alpha Qf)-\lambda Pf-\lambda Qf=Pg+Qg
$$
Which gives
\begin{eqnarray}\label{s-1}
Qf &=&-\frac{1}{\lambda}Qg,
\end{eqnarray}
\begin{eqnarray}\label{s-2}
P(\bar\alpha Pf)-\lambda Pf
=Pg-P(\bar\alpha Qf).
\end{eqnarray}
Using (\ref{s-1}) in (\ref{s-2}), we get
$$(T_{\bar{\alpha}}-\lambda I)Pf= Pg+\frac{1}{\lambda}P(\bar\alpha Qg).$$
Hence the necessary part of the lemma is proved. To prove the sufficient part, first note that, the given conditions clearly imply (\ref{s-1}) and (\ref{s-2}). Now, adding these two equations we get that $(S_{\alpha,0}^*-\lambda I)f=g$.
\end{proof}
%For $\alpha\in L^\infty$, recall the definition of $\widetilde{T}_\alpha$ from section 
%3. It is an operator on $H^{{2}^{\perp}}$ 
%given by $$\widetilde{T}_\alpha (f)  = Q(\alpha f), ~~~f \in 
%H^{{2}^{\perp}} \mbox{and}~~~ \alpha \in L^{\infty}. $$
%\begin{lem}
%Let $\alpha\in L^\infty$. Then $\sigma(\widetilde{T}_{\alpha})=\sigma(T_{\bar\alpha})$
%\end{lem}
%\begin{proof}
%Let $\widetilde{T}_\alpha$ is injective. If $T_{\bar\alpha}f-\lambda f=0$ for some $f\in H^2$ then $P(\bar\alpha f)-\lambda Pf=0$ and $\lambda Qf=0$  
%\end{proof}

\begin{defn} For $\alpha \in L^\infty,$ the \textit{essential range} of $\alpha$ is defined to be $$ \mbox{ess ran}~~ \alpha = \{ \lambda : m \{ e^{i\theta} : | \alpha (e^{i\theta}) - \lambda| < \epsilon\} > 0, \mbox{for all}~~ \epsilon > 0\},$$ where $m$ is the normalized Lebesgue measure.
\end{defn}
The following theorem gives some descriptions about the spectrum for certain special cases.
\begin{thm}\label{spectrum}
Let $\alpha,\beta\in L^{\infty}$.

\noindent\textup{(i)} \textup{ess ran $\alpha$} $\cup$ \textup{ess ran $\beta$} $\subset\Pi(S_{\alpha,\beta})$

\noindent \textup{(ii)} $\sigma(S_{\alpha,0})=\sigma(T_\alpha)\cup\{0\}$ and $\sigma(S_{0,\beta})=\sigma(\tilde{T}_\beta)\cup\{0\}$.

\noindent \textup{(iii)} If $\alpha$ is analytic or  $\beta$ is co-analytic, then $\sigma(S_{\alpha,\beta})\subset \sigma(T_\alpha)\cup\sigma(\tilde{T}_\beta)\cup\{0\}$.

\noindent {(iv)} If $\alpha$ is analytic and $\beta$ is co-analytic, then
$\sigma(S_{\alpha,\beta})\cup\{0\}= \sigma(T_\alpha)\cup\sigma(\tilde{T}_\beta)\cup\{0\}$.
\end{thm}
\begin{proof}
(i) Let $M_{\alpha}$ denotes the usual multiplication operator on $L^2$ given by $M_{\alpha} (f) = \alpha f.$ Let $\lambda\in \textup{ess ran}~~  \alpha$. Then (see the proof of Theorem 3.3.1, page-108, \cite{MR}) there is a sequence of unit norm functions $h_n\in H^2$ such that $||(M_{\alpha}-\lambda)h_n||\rightarrow 0$. But, note that, $h_n$ being in $H^2$, $(M_{\alpha}-\lambda)h_n=(S_{\alpha,\beta}-\lambda)h_n$. Therefore $||(S_{\alpha,\beta}-\lambda)h_n||\rightarrow 0$ proving that $\lambda\in\Pi(S_{\alpha,\beta})$. Hence \textup{ess ran $\alpha$} $\subset\Pi(S_{\alpha,\beta})$. Now let $\lambda\in\textup{ess ran}$ $\beta$. Looking at the same proof (i.e. proof of Theorem 3.3.1, page-108, \cite{MR}) it is not hard to see that there is an sequence of unit functions $g_n$ in $H^{2\bot}$ such that $||(M_\beta-\lambda)g_n||\rightarrow 0$. But, $g_n$ being in $H^{2\bot}$,  $(M_\beta-\lambda)g_n=(S_{\alpha,\beta}-\lambda)g_n $. It follows that $\lambda\in\Pi(S_{\alpha,\beta})$. Therefore \textup{ess ran $\beta$} $\subset\Pi(S_{\alpha,\beta})$. This completes the proof of (i).

(ii) Clearly $0$ is in the spectrum of $S_{\alpha , 0}$. So, we have to prove that non zero spectrum of $S_{\alpha ,0}$ is same as that of the Toeplitz operator $T_\alpha$. Since spectrum of the adjoint of an operator is same as the conjugate of the spectrum of the operator, it is enough to show that non zero spectrum of $S^*_{\alpha0}$ is same as that of $T^*_\alpha=T_{\bar\alpha}$. So take $\lambda\neq 0$. We have to show that $S_{\alpha , 0}^*-\lambda I$ is invertible iff $T_{\bar\alpha}-\lambda I$ is invertible. First we note that $(S^*_{\alpha, 0}-\lambda I)f=(T_{\bar\alpha}-\lambda I)f$ for all $f\in H^2$.

If $S^*_{\alpha, 0}-\lambda I$ is injective then $(S^*_{\alpha, 0}-
\lambda I)f\neq 0$ for all non zero $f$ in $L^2$ which implies that
$(T_{\bar\alpha}-\lambda I)f\neq 0$ for all non zero $f$ in $H^2$ and hence $T_{\bar\alpha}-\lambda I$ is injective.

If $S^*_{\alpha , 0}-\lambda I$ is onto then for any $g\in H^2$ there is a $f\in L^2$ such that $(S_{\alpha, 0}^*-\lambda I)f=g$. But then by Lemma \ref{l-1}, $f\in H^2$. Therefore $(T_{\bar\alpha}-\lambda I)f=g$. Hence $T_{\bar\alpha}-\lambda I$ is onto.

Let   $T_{\bar\alpha}-\lambda I$ is injective. Then we shall show that  $S^*_{\alpha, 0}-\lambda I$ is also injective.
Let $(S^*_{\alpha, 0}-\lambda I) f = 0.$ Then $P(\bar{\alpha}f) - \lambda f = 0.$ Applying  $Q$ to both sides we have $Qf = 0.$ Which implies that $f \in H^2.$ As a result we have $(T_{\bar\alpha}-\lambda I) f = 0$ and which implies $f = 0$ proving the injectivity of  $S^*_{\alpha, 0}-\lambda I.$

Now assume that $T_{\bar\alpha}-\lambda I$ is onto. Let $g\in L^2$. There exist $h\in H^2$ such that
$$(T_{\bar\alpha}-\lambda I)h=Pg+\frac{1}{\lambda}P(\bar\alpha Qg).
$$
Define $ f\in L^2$ by $Qf=-\frac{1}{\lambda}Qg$ and $Pf=h$. Then, by Lemma \ref{l-1}, $(S^*_{\alpha, o}-\lambda I)f=g$ proving the ontoness of $S^*_{\alpha, 0}-\lambda I$. This finishes the proof of the fact that $\sigma(S_{\alpha, 0})=\sigma(T_\alpha)\cup\{0\}$. The proof of $\sigma(S_{0, \beta})=\sigma(\tilde{T}_\beta)\cup\{0\}$ is similar.

(iii) Let $\beta\in H^{2\bot}$. In view of (ii) it is enough to show that $\sigma(S_{\alpha, \beta})\subset \sigma(S_{\alpha, 0})\cup\sigma(S_{0,\beta })$ which is again equivalent to showing that $\sigma(S^*_{\alpha,\beta})\subset \sigma(S^*_{\alpha, 0})\cup\sigma(S^*_{0,\beta})$. Let $\lambda\neq 0$ be such that both $S^*_{\alpha, 0}-\lambda I$ and $S^*_{0,\beta}-\lambda I$ are invertible. We need to show that $S_{\alpha,\beta}^*-\lambda I$ is invertible.  If $(S^*_{\alpha,\beta}-\lambda )f=0$, writing $f=Pf+Qf$, and then comparing the $H^2$, $H^{2\bot}$ components we get
\begin{eqnarray}\label{s-3}
P(\bar\alpha Pf)+P(\bar\alpha Qf)-\lambda Pf=0,
\end{eqnarray}
$$
Q(\bar\beta Pf)+Q(\bar\beta Qf)-\lambda Qf=0.
$$
Since $\beta$ is co-analytic, the second equation gives
$
(S^*_{0,\beta}-\lambda I)Qf=0.
$
But $(S^*_{0,\beta}-\lambda I)$ being injective, $Qf=0$. Therefore (\ref{s-3}) gives $(S^*_{\alpha, 0}-\lambda I)Pf=0$
which, by the injectivity of $S^*_{\alpha, 0}-\lambda I$, implies that $Pf=0$. Hence $f=0$. Therefore $(S^*_{\alpha,\beta}-\lambda I)$ is injective. To prove its ontoness, let $g\in L^2$. Since $S^*_{0,\beta}-\lambda I$ is onto, there exists $h_1\in L^2$ such that $(S^*_{0,\beta}-\lambda I)h_1=Qg$. Applying the operator $P$, we get that $Ph_1=0$ or $h_1\in H^{2\bot}$. Since $S_{\alpha, 0}^*-\lambda I$ is onto, there exist $h_2\in L^2$ such that $(S^*_{\alpha, 0}-\lambda I)h_2=Pg-P(\bar\alpha h_1)$. Applying $Q$ to both side, we see that $Qh_2=0$ or $h_2\in H^2$. Now define $f=h_1+h_2$ so that $Qf=h_1$ and $Pf=h_2$. With this definition of $f$ it is not hard to see that
\begin{eqnarray}\label{s-4}
P(\bar\alpha Pf)+P(\bar\alpha Qf)-\lambda Pf=Pg,
\end{eqnarray}
$$
Q(\bar\beta Qf)-\lambda Qf=Qg.
$$
Since $\beta$ is co-analytic, the last equation is equivalent to

\begin{eqnarray}\label{s-5}
Q(\bar\beta Pf)+Q(\bar\beta Qf)-\lambda Qf=Qg.
\end{eqnarray}
Adding (\ref{s-4}) and
(\ref{s-5}) we get $(S^*_{\alpha,\beta}-\lambda I)f=g$. Therefore $S^*_{\alpha,\beta}-\lambda I$ is onto. This completes the proof of (ii) when $\beta$ is co-analytic. The proof for analytic $\alpha$ is similar.

(iv) In view of (i) and (ii), it is enough to show that $\sigma(S_{\alpha, 0})\cup\sigma(S_{0,\beta})\subset \sigma(S_{\alpha, \beta})\cup\{0\}$. So let $\lambda\notin \sigma(S_{\alpha,\beta})\cup\{0\}$. We have to show that $\lambda\notin\sigma(S_{\alpha, 0})$ and $\lambda\notin\sigma(S_{0, \beta})$. If $(S_{\alpha, 0}-\lambda I)f=0$, then $\alpha$ being analytic, $Qf=0$ and $(\alpha-\lambda)Pf=0$. Again $(S_{\alpha}-\lambda I)Pf=(\alpha-\lambda)Pf=0$. Since $S_{\alpha,\beta}$ is injective, $Pf=0$. hence $f=0$. Therefore $(S_{\alpha, 0}-\lambda I)$ is injective. Next, let $g\in L^2$. There exists $h\in L^2$ such that $(S_{\alpha,\beta}-\lambda I)h=Pg$. Applying the operator $P$ both side we get $(\alpha-\lambda)Ph=Pg$. Now define $f\in L^2$ by $Pf=Ph$ and $Qf=-\frac{1}{\lambda}Qg$. With this definition of $f$, it is easy to see that $(S_{\alpha, 0}-\lambda I)f=g$ proving the ontoness of $S_{\alpha, 0}-\lambda I$. So we have proved that $S_{\alpha, 0}-\lambda I$ is invertible and hence $\lambda\notin \sigma(S_{\alpha, 0})$. In a similar way we can prove that $\lambda\notin \sigma(S_{0,\beta})$.
\end{proof}
Like the Toeplitz operator, the spectrum of $S_{\alpha,\beta}$ can be known completely provided $\alpha$, $\beta$ are continuous functions. Recall the definition of index of a continuous curve $\phi:S^1\rightarrow\mathbb{C}$ at a point $a\notin\textup{range}\,\phi$ :
$$
\textup{ind}_a\phi=\frac{1}{2\pi i}\int_\phi\frac{1}{z-a}dz.
$$
The index is also called the winding number of $\phi$ around $a$.
% because it counts the number of circling of the curve $\phi$ around the point $a$.
We have the following theorem giving complete information about the spectrum of $S_{\alpha,\beta}$, when $\alpha$ and $\beta$ are continuous.
\begin{thm}\label{spectrum when alpha and beta are continuous}
Let both $\alpha$ and $\beta$ be continuous. Then
$$
\sigma(S_{\alpha,\beta})=\textup{range}\,\alpha\cup \textup{range}\,\beta\cup\left\{a\notin (\textup{range}\,\alpha\cup \textup{range}\,\beta):\textup{ind}_a\alpha\neq\textup{ind}_a\beta\right\}.
$$
\end{thm}
The proof of the above theorem is an adaptation of the method used to prove the corresponding result of Toeplitz operator (see Theorem 3.3.18, page-116, \cite{MR}). We need several lemmas.
\begin{lem}\label{invertibility when alpha and beta....}
Let $\alpha\in H^\infty$, $\beta\in\overline{H^\infty}$  such that their inverses exist and $\alpha^{-1}\in H^\infty$, $\beta^{-1}\in\overline{H^\infty}$.  Then $S_{\alpha,\beta}$ is invertible. 
\end{lem}
\begin{proof}
The proof follows from the fact that $S_{\alpha,\beta}f=g$ iff $Pf=\alpha^{-1}Pg$ and $Qf=\beta^{-1}Qg$; in other words, $S^{-1}_{\alpha,\beta}=S_{\alpha^{-1},\beta^{-1}}$.
\end{proof}

\begin{lem}\label{invertibility of S-z-nalpha,1}
Let $\alpha\in L^\infty$ be such that $S_{\alpha,1}$ is invertible. Let $n$ be an integer. Then $S_{z^n\alpha,1}$ is invertible iff $n=0$.
\end{lem}
\begin{proof}
If part is obvious. So we need to show that if $n\neq 0$, then $S_{z^n\alpha,1}$ is not invertible.

Case-1 : $n> 0$. By Theorem \ref{thm-multiplicati is of the same form}, $S_{z^n\alpha,1}=S_{z^n,1}S_{\alpha,1}$. Note that $1$ has no pre-image under $S_{z^n,1}$. Therefore $S_{z^n,1}$ is not invertible, and hence the same is true for $S_{z^n\alpha,1}$.

Case-2 : $n<0$. Write $k$ for $-n$ so that $s_{z^n\alpha,1}=S_{\bar z^k\alpha,1}$. Let $P_{k-1}$ denotes the projection onto the space $\{f\in L^2:\widehat{f}(m)=0\,\,\textup{for all}\,m\neq 0,1,\cdots,(k-1)\}$. Then a small calculation shows that
$$
S_{\bar z^k\alpha,1}-S_{\alpha,1}S_{\bar z^k,1}=\bar z^k(\alpha-1)P_{k-1}.
$$
Let $g=-S_{\alpha,1}^{-1}\left(\bar z^k(\alpha-1)\right)$. Now, it is not hard to see that image of the space $\{f\in L^2:\widehat{f}(0)=1,\widehat{f}(1)=\cdots=\widehat{f}(k-1)=0\}$ under $S_{\bar z^k,1}$ is the full space $L^2$. Therefore there is an $f\in L^2$ with $\widehat{f}(0)=1,\widehat{f}(1)=\cdots=\widehat{f}(k-1)=0$ such that $s_{\bar z^k,1}f=g$. So, the above equation, when applied on $f$, gives $S_{\bar z^k\alpha,1}f=0$. Hence $S_{\bar z^k\alpha,1}$ is not invertible.
\end{proof}
\begin{lem}\label{spectrum of phi}
Let $\phi$ be a continuous function which does not take the value $0$. Then $S_{\phi,1}$ is invertible iff\, $\textup{ind}_0\,\phi =0$.
\end{lem}
\begin{proof}
Since $\phi$ is continuous and it does not take the value $0$, there is a positive number $\delta$ such that $|\phi(e^{i\theta})|>\delta$ for all $0\leq\theta\leq 2\pi$. Certainly we can choose a trigonometric polynomial $p$ such that $|\phi(e^{i\theta})-p(e^{i\theta})|<\frac{\delta}{3}$ for all $\theta$. Clearly $|p(e^{i\theta})|>\frac{2\delta}{3}$ for all $\theta$. We define the continuous function $\psi :=\frac{\phi-p}{p}$ so that $\phi=p(1+\psi)$ and $|\psi(e^{i\theta})|<\frac{1}{2}$. Therefore the continuous function $1+\psi$ can not wind around $0$ so that $\textup{ind}_0\,(1+\psi)=0$, and hence $\textup{ind}_0\,\phi=\textup{ind}_0\, p+\textup{ind}_0 \,(1+\psi)=\textup{ind}_0\, p$. Therefore it is enough to show that $S_{\phi,1}$ is invertible if and only if $\textup{ind}_0\,p=0$. Let $n$ be a non negative number such that $e^{in\theta}p(e^{i\theta})$ has no non-zero Fourier coefficients corresponding to negative indices. Since $p$ has no zero on $S^1$, we can factored $e^{in\theta}p(e^{i\theta})$ as
$$
e^{in\theta}p(e^{i\theta})=ce^{im\theta}\prod_{1}^k(e^{i\theta}-z_j)\prod_1^{l}(e^{i\theta}-w_j)
,$$
for some non negative integers $m,k,l$, non zero complex numbers $z_j$'s and $w_j$'s, where $z_j$'s lie (strictly) inside the circle $S^1$ and $w_j$'s  lie (strictly) outside the circle.
We can write
$$
p(e^{i\theta})=ce^{i(m-n+k)\theta}\prod_{1}^k(1-z_je^{-i\theta})\prod_1^{l}(e^{i\theta}-w_j).
$$
Since $|w_j|>1$ and $|z_j|<1$, it follows that $\textup{ind}_0\,(e^{i\theta}-w_j)=0$ and  $\textup{ind}_0\,(1-z_je^{-i\theta})=0$. So, from the above equation, we get $\textup{ind}_0\,p=m-n+k$. Therefore it is enough to show that $S_{\phi,1}$ is invertible if and only if $m-n+k= 0$.

We write $p$ as 
$$
p(e^{i\theta})=ce^{i(m-n+k)\theta}u(e^{i\theta})v(e^{i\theta}),
$$
where
$$
u(e^{i\theta})=\prod_1^{l}(e^{i\theta}-w_j),\,\,\,\,v(e^{i\theta})=\prod_{1}^k(1-z_je^{-i\theta}).
$$
Note that $u$ is analytic, $v$ is co-analytic and both of them are continuous. Since $|w_j|>1$ for all $j=1,2,\cdots l$, it follows that $u^{-1}$ is analytic and continuous.  Similarly, $v^{-1}$ is co-analytic and continuous. Now, it is not hard to see that,
$
S_{\phi,1}=S_{cz^{m-n+k}(1+\psi)uv,1}$ is invertible iff $
S_{cz^{m-n+k}(1+\psi)u,v^{-1}}$ is invertible. But, by Theorem \ref{thm-multiplicati is of the same form},  $S_{cz^{m-n+k}(1+\psi)u,v^{-1}}=S_{z^{m-n+k}(1+\psi),1}S_{cu,v^{-1}}$, and by Lemma \ref{invertibility when alpha and beta....}, $S_{cu,v^{-1}}$ is invertible. Therefore we can say that $S_{\phi,1}$ is invertible iff $S_{z^{m-n+k}(1+\psi),1}$ is invertible. Hence the proof follows by Lemma \ref{invertibility of S-z-nalpha,1} if we can show that $S_{(1+\psi),1}$ is invertible. But this is true since
$$
||I-S_{(1+\psi),1}||=||S_{\psi,0}||\leq||\psi||_\infty\leq\frac{1}{2}<1.
$$
\end{proof}
Now we prove Theorem \ref{spectrum when alpha and beta are continuous}.
\begin{proof}(proof of Theorem \ref{spectrum when alpha and beta are continuous}.) Theorem \ref{spectrum} (i) implies that $\textup{range}\,\alpha\cup\textup{range}\,\beta$ is contained in $\sigma(S_{\alpha,\beta})$. Thus we only need to show that, for $a\notin\textup{range}\,\alpha\cup\textup{range}\,\beta$,\, $S_{\alpha,\beta}-aI$ is invertible iff $\textup{ind}_a\,\alpha=\textup{ind}_a\,\beta$. Since $a\notin\textup{range}\,\alpha\cup\textup{range}\,\beta$ iff $0\notin\textup{range}\,(\alpha-a)\cup\textup{range}\,(\beta-a)$, $S_{\alpha,\beta}-aI=S_{\alpha-a,\beta-a}$, and $\textup{ind}_a\,\alpha=\textup{ind}_0\,(\alpha-a)$, $\textup{ind}_a\,\beta=\textup{ind}_0\,(\beta-a)$, we may assume that $a=0$. Therefore we need to show the following : If $\alpha,\beta$ are never vanishing continuous functions then $S_{\alpha,\beta}$ is invertible iff $\textup{ind}_0\,\alpha=\textup{ind}_0\,\beta$. Since $\beta$ is never vanishing continuous function we can show 
that $S_{\alpha,\beta}$ is invertible iff $S_{\frac{\alpha}{\beta},1}$ is invertible. So writing $\phi=\frac{\alpha}{\beta}$, using the fact that $\textup{ind}_0\frac{\alpha}{\beta}=\textup{ind}_0\,\alpha-\textup{ind}_0\,\beta$, our problem ultimately reduces to Lemma \ref{spectrum of phi}. Hence the proof.
\end{proof}

\begin{rem}
Theorem \ref{spectrum when alpha and beta are continuous}, in particular, implies that if $\alpha,\beta$ are continuous then the spectral radius of $S_{\alpha,\beta}$ is $\textup{max}\,\{||\alpha||_\infty,||\beta||_{\infty}\}$. But we don't know whether the same is true for general $\alpha,\beta$. 
\end{rem}
\section{Injectivity}
In this section we discuss about the injectivity of the operators $S_{\alpha,\beta}$ and $S^*_{\alpha,\beta}.$ Unlike the Toeplitz operator, the Coburn Alternative type theorem is not true for general $S_{\alpha,\beta}$.  

For a measurable set $A \subset S^1,$ $|A|$ denotes its usual Lebesgue measure.

\begin{thm}\label{coburn}
Let $\alpha$ and $\beta$ be non zero $L^\infty$ functions. Let $Z_\alpha$ and $Z_\beta$ denote the zero sets of $\alpha$ and $\beta$ respectively. 

\noindent \textup{(i)} If $|Z_\alpha|=0$ or $|Z_\beta|=0$ then at least one of $S_{\alpha,\beta}$ and $S^*_{\alpha,\beta}$ is injective.

\noindent \textup{(ii)} Let $|Z_\alpha|,|Z_\beta|\neq 0.$ If  $Z_\alpha\neq Z_\beta$ (in the sense of measure i.e. $|Z_\alpha\setminus Z_\beta|\neq 0$) then $S_{\alpha,\beta}$ is injective

\noindent \textup{(iii)} Let $|Z_\alpha|,|Z_\beta|\neq 0.$ $S^*_{\alpha,\beta}$ is injective iff $|Z_\alpha\cap Z_\beta|=0$.
\end{thm}
\begin{proof}
(i) Without loss of generality assume that $|Z_\alpha|=0$. Let $S_{\alpha,\beta}f=0$ and $S^*_{\alpha,\beta}g=0$ for some $f,g\in L^2$. We have to show that at least one of $f$ and $g$ is zero. Now $S^*_{\alpha,\beta}g=0$ implies that $P(\bar\alpha g)+Q(\bar \beta g)=0$ or $P(\bar\alpha g)=0$ and $Q(\bar \beta g)=0$ which is equivalent to saying that $\bar\alpha g$ in $H^{2\bot}$ and $\bar\beta g$ is in $H^2$. On the other hand, since $S_{\alpha,\beta}f=0$, we have $\alpha Pf+\beta Qf=0$. Taking conjugate and then multiplying by $g$ we get $g\bar\alpha\overline{Pf}+g\bar\beta\overline{Qf}=0$. Since $g\bar\alpha$ is in $H^{2\bot}$ and $\overline{Pf}$ is in $\overline{ H^2}$,  $g\bar\alpha\overline{Pf}$ is an $L^1$ function whose all the non negative Fourier coefficients are zero. Similarly $g\bar\beta\overline {Qf}$ is an $L^1$ function whose all the non positive Fourier coefficients are zero. Therefore we can conclude that $g\bar\alpha\overline{Pf}=g\bar\beta\overline{Qf}=0$. If $g$ is non zero then, $|Z_\alpha|$ being zero, $g\bar\alpha$ is non zero; therefore $\overline{Pf}=0$ or $Pf=0$ and hence $\beta Qf=0$ (as $\alpha Pf+\beta Qf=0$) so that $Qf=0$ and thus $f=0$.

(ii) If $Z_\alpha\neq Z_\beta$ then there is a set $E$ of positive measure such that one of $\alpha$ and $\beta$ is zero on $E$ and other is non zero on each point of $E$. Assume that $\alpha$ is zero on $E$ and $\beta$ never vanishes on $E$. Now, if $S_{\alpha,\beta}f=0$ i.e. $\alpha Pf+\beta Qf=0$ then clearly $Qf$ is zero almost every where on $E$. Therefore $Qf$ must be identically zero and consequently $\alpha Pf=0$ or $Pf=0$ and thus $f=0$ proving the injectivity of $S_{\alpha,\beta}$. %Now let $Z_{\alpha}=Z_\beta$. We have to show that $S_{\alpha,\beta}$ is not injective............................................

(iii) Let $|Z_\alpha\cap Z_\beta|=0$. If $S^*_{\alpha,\beta}g=0$  i.e. $P(\bar \alpha g)+Q(\bar\beta g)=0$ then $P(\bar\alpha g)=0$ and $Q(\bar\beta g)=0$ or equivalently we can say $\bar\alpha g$ is in $H^{2\bot}$ and $\bar\beta g$ in $H^2$. Since $\bar\alpha g$ is zero on $Z_\alpha$ which has positive measure, $\bar\alpha g=0$. Similarly $\bar\beta g=0$. Since $|Z_\alpha\cap Z_\beta|=0$ we must have $g=0$. Therefore $S^*_{\alpha,\beta}$ is injective. Now let $|Z_\alpha\cap Z_\beta|\neq 0$. If we take $g$ to be the
indicator function of $Z_\alpha\cap Z_\beta$, it is easy to see that $S^*_{\alpha,\beta}g=0$ proving the non injectivity of $S^*_{\alpha,\beta}$.
\end{proof}

\begin{rem}
The above theorem would be complete if we could prove the other implication of (ii) i.e if $|Z_\alpha|,|Z_\beta|\neq 0$ and $Z_\alpha= Z_\beta$ then $S_{\alpha,\beta}$ is not injective. But we don't know whether this is true in general. Whatever example we got, we saw that the result is true. We discus one such special case here : Let $E$ be an open non empty interval in $S^1$. Let $\alpha,\beta\in L^\infty$ be such that both are never vanishing continuous functions on $ E^c$ and $Z_\alpha=Z_\beta=E$. It is geometrically evident that we can define never vanishing continuous functions $\alpha^\prime$, $\beta^\prime$ on $S^1$ such that they are same as $\alpha$, $\beta$ respectively on $E^c$, as well as their winding number at $0$ are same i.e. $\textup{ind}_0\,\alpha^\prime=\textup{ind}_0\,\beta^\prime$. Therefore, by Theorem \ref{spectrum when alpha and beta are continuous}, $S_{\alphaì\prime,\beta^\prime}$ is invertible. Hence $S_{\alpha^\prime,\beta^\prime}f=\chi_{E}$ for some non zero $f\in L^2$; where $\chi_E$ denotes the indicator function of $E$. Since $S_{\alpha,\beta}=\chi_{E^c}S_{\alpha^\prime,\beta^\prime}$, $S_{\alpha,\beta}f=0$ implying the non injectivity of $S_{\alpha,\beta}$.
\end{rem}
Like the Toeplitz operator, (i) of the above theorem has the following applications.
\begin{cor}
Let $\alpha,\beta$ be non zero $L^\infty$ functions. Assume that either $|Z_\alpha|=0$ or $|Z_\beta|=0$
% at least one of $\alpha$ and $\beta$ does not vanish on a set of positive measure. 
Then $S_{\alpha,\beta}$ has dense range if it is not injective.
\end{cor}

\begin{proof}
If $S_{\alpha,\beta}$ is not injective, by Theorem \ref{coburn}, (i), $S^*_{\alpha,\beta^*}$ is injective. If possible let the range of $S_{\alpha,\beta}$ is not dense. Then there there is a non zero $g$ such that $\langle S_{\alpha,\beta}f,g\rangle=0$ for all $f$. Therefore $\langle f,S_{\alpha,\beta}^*g\rangle=0$ for all $f$. Putting $f=S_{\alpha,\beta}^*g$, we conclude that $S_{\alpha,\beta}^*g=0$ which contradicts the injectivity of $S_{\alpha,\beta}^*$.
\end{proof}

\begin{cor}
Let $\alpha,\beta$ be non-constant $L^\infty$ functions. Assume that at least one of $\alpha$ and $\beta$ can not be constant on any set of positive measure. Then
$$
\Pi_0(S_{\alpha,\beta})\cap\overline{\Pi_0(S_{\alpha,\beta}^*)}=\emptyset.
$$
\end{cor}
\begin{proof}
If possible let  $\lambda\in\Pi_0(S_{\alpha,\beta})$ and $\bar\lambda\in\Pi_0(S_{\alpha,\beta}^*)$ for some complex number $\lambda$. Then there are non zero functions $f$ and $g$ such that $(S_{\alpha,\beta}-\lambda I)f=0$ and $(S_{\alpha,\beta}^*-\bar\lambda I)g=0$. But these are equivalent to saying that $S_{\alpha-\lambda,\beta-\lambda}f=0$ and $S^*_{\alpha-\lambda,\beta-\lambda}g=0$ so that both $S_{\alpha-\lambda,\beta-\lambda}$ and $S^*_{\alpha-\lambda,\beta-\lambda}$ are not injective contradicting (i) of Theorem \ref{coburn}.
\end{proof}

\begin{cor}
Let $\alpha,\beta$ are non-constant real-valued $L^\infty$ functions such that at least one of $\alpha$ and $\beta$ can not be constant on any set of positive measure. Also assume that $\alpha-\beta$ is a real constant. Then point spectrum of $S_{\alpha,\beta}$ is empty.
\end{cor}

\begin{proof}
By Theorem 2.1 in \cite{NY8}, $S_{\alpha,\beta}$ is self-adjoint. If possible let $\lambda\in \Pi_0(S_{\alpha,\beta})$. Since $S_{\alpha,\beta}$ is self adjoint $\lambda$ must be real and it is also in the point spectrum of $S_{\alpha,\beta}^*$. This implies a contradiction to the previous corollary.
\end{proof}

\end{document}